\documentclass[a4paper,twoside,12pt]{article}
\usepackage{amssymb,amsfonts,amsmath,amsthm,latexsym}
\usepackage[pdftex,bookmarks,colorlinks=false]{hyperref}
\usepackage[hmargin=1.25in,vmargin=1.25in]{geometry}
\usepackage[auth-sc-lg,affil-sl]{authblk}
\newtheorem{thm}{Theorem}[section]
\newtheorem{cor}[thm]{Corollary}
\newtheorem{defn}[thm]{Definition}
\newtheorem{lem}[thm]{Lemma}

\newtheorem{prop}[thm]{Proposition}

\def\ni{\noindent}
\def\N{\mathbb{N}}

\title{\textbf{\sc A study on the Primitive Holes of Certain Graphs}}

\author{Johan Kok}
\affil{\small Tshwane Metropolitan Police Department\\ City of Tshwane, Republic of South Africa\\ E-mail: kokkiek2@tshwane.gov.za}

\author{N. K. Sudev}
\affil{\small Department of Mathematics\\ Vidya Academy of Science \& Technology \\ Thalakkottukara, Thrissur - 680501, India.\\ E-mail: sudevnk@gmail.com}

\date{}

\begin{document}
\maketitle

\begin{abstract}
A hole of a simple connected graph $G$ is a chordless cycle $C_n$, where $n\in \N,n\ge 4$, in the graph $G$. The girth of a simple connected graph $G$ is the smallest cycle in $G$, if any such cycle exists. It can be observed that all such smallest cycles are necessarily chordless. We call the cycle $C_3$ in a given graph $G$ a primitive hole of that graph. We introduce the notion of the primitive hole number of a graph as the number of primitive holes present in that graph. In this paper, we determine the primitive hole number of certain standard graphs. Also, we determine the primitive hole number of the underlying graph of a Jaco graph, $J_{n+1}^{\ast}(1)$, where $n\in N,n\ge 4$ recursively in terms of the underlying Jaco graph $J_n(1)$, with prime Jaconian vertex $v_i$.  The notion of primitive degree of the vertices of a graph is also introduced and the primitive degree of the vertices of certain graphs is also determined in this paper.
\end{abstract}

\ni \textbf{Key Words:} Jaco graph, primitive hole, primitive hole number, girth of a graph, primitive degree of a vertex. 
\vspace{0.2cm}

\ni \textbf{Mathematics Subject Classification:} 05C07, 05C20, 05C38.

\section{Introduction}

For general notations and concepts in graph theory, we refer to \cite{BM}, \cite{FH} and \cite{DBW} and for digraph theory, we further refer to \cite{CL1} and \cite{JG1}. All graphs mentioned in this paper are simple, connected and finite graphs, unless mentioned otherwise.

A \textit{hole} of a simple connected graph $G$ is a chordless cycle $C_n$ , where $n \in  N, n\ge 4$, in $G$. The \textit{girth} of a simple connected graph $G$, denoted by $g(G)$, is the order of the smallest cycle in $G$. It is to be noted that such smallest cycles are necessarily chordless. In this paper, contrary to the usual conventions, we propose that the girth of an acyclic graph is $0$, which enables us to explore evolutionary \textit{hole growth} like, a hole $C_k$ may grow over time units $t \in X\subseteq \N$ over a integer valued function $x(t) = j$ to attain $j$ additional cyclic vertices at $t$. Hence $g(\lim\limits_{t\to \infty}(C_{k+x(t)})) = \infty$. It also implies that for simple connected graphs $G_1, G_2, G_3,\ldots, G_n$, $g(\bigcup\limits _{i=1}^{n} G_i) = \sum\limits_{i=1}^{n}g(G_i)$ and allows quite naturally that, $g(\bigcup\limits _{i=1}^{\infty} G_i) = \sum\limits_{i=1}^{\infty}g(G_i) = \infty$.  These conventions reconcile the inherent conflict between the definitions of a girth and hole.

\section{Primitive Hole Number of Graphs}

In this section, we introduce the notion of primitive holes and the primitive hole number of a given graph as follows.

\begin{defn}{\rm 
A \textit{primitive hole} of a graph $G$ is a triangle $C_3$ in it. The \textit{primitive hole number} of a simple connected graph $G$, denoted $h(G)$, is the number of {\em primitive holes} in $G$.}
\end{defn} 

If a simple connected graph $G$ has no primitive hole, then we say that $h(G)=0$. Hence, it follows that for simple connected graphs $G_1, G_2, G_3, \ldots, G_n$, $h(\bigcup\limits _{i=1}^{n} G_i) = \sum\limits_{i=1}^{n}h(G_i)$.
\\

\ni In view of the above definitions, we can establish the following theorem.

\begin{prop}\label{P-PHN-KN}
The number of primitive holes in a complete graph $K_n$ is $\binom{n}{3}$.
\end {prop}
\begin{proof} 
A primitive hole of any given graph $G$ is a triangle $K_3$. Hence, $h(K_n)$ is the the number of distinct triangles in $K_n$. It can be noticed that every distinct triplet of vertices in $K_n$ forms a triangle in $K_n$. Therefore, the number of primitive holes in $K_n$ is equal to the number of ways in which three vertices can be chosen from the given set of $n$ vertices. Hence, $h(K_n)=\binom{n}{3}$.
\end{proof}

The relevance of the observation is that the identification algorithm might assist in determining $h(G)$ in general.

The following result establishes a relation between the primitive hole numbers of a given graph and its subgraphs.

\begin{prop}\label{P-PHN-SG}
For any subgraph $H$ of a given graph $G$, $h(H)\le h(G)$. 
\end{prop}
\begin{proof}
Let $G$ be a given graph and $H$ be its non-trivial subgraph. Then, we have either $V(H)\subseteq V(G)$ and/or $E(H)\subseteq E(G)$. Then we have the following cases. 

\ni {\em Case-1:} First assume that $V(H)$ is a non-trivial subset of $V(G)$. Then, there exists some vertex, say $v$, in $V(G)$ but not in $V(H)$. If $v$ is a pendant vertex of $G$, then it is not in any triangle of $G$ and hence the number of triangles in $G-v$ and $G$ are the same. If $v$ is not a pendant vertex of $G$, then $v$ must be adjacent to at least two vertices in $G$. Let $u$ and $w$ be two vertices that are adjacent to $v$ in $G$. If $u$ and $v$ are adjacent vertices in $G$, then the vertices $u,v,w$ form a triangle in $G$ and this triangle will be missing in $G-v$. If $u$ and $v$ are non-adjacent in $G$, then also $G$ and $G-v$ contain same number of primitive holes.

\vspace{0.2cm}

\ni {\em Case-2:} If $H$ is a spanning subgraph of $G$, then $V(H)=V(G)$. In this case, $E(H)$ is a non-trivial subset of $E(G)$. Then, there exists some edge in $G$ that is not in $H$. Let $e$ be an edge in $G$ that is not in $H$. If this edge is in a triangle of $G$, then as explained in Case-1, that triangle will be missing in $H$. Otherwise, the number of triangles in $G$ and $G-e$ are the same. 

\ni Invoking all the above cases, we have $h(H)\le h(G)$. 
\end{proof}

\ni Invoking the above results, we have the following theorem.  

\begin{thm}
For any simple, connected graph $G$ on $n$ vertices, $0\le h(G)\le \binom{n}{3}$.
\end{thm}
\begin{proof}
The result follows as an immediate consequence of Proposition \ref{P-PHN-KN} and Proposition \ref{P-PHN-SG}.
\end{proof}

For a given graph $G$, either if $G$ is an acyclic graph or a $g(G)\ge 4$, then $h(G)=0$. Then, we need only consider the graphs whose girth is $3$. If $g(G)=3$, then the graph $G$ contains at least one primitive hole. The following lemma explains a relation between the size of a graph $G$ and the number of primitive holes in $G$.

\begin{lem}
For any simple connected graph $G$ we have that, $h(G)\le|E(G)|$.
\end{lem}
\begin{proof}
It is to be noted that any two distinct triangles in $G$ can have at most one edge in common. Hence, if $t$ is the number of triangles in $G$, then the minimum number of edges in $G$ must be $2t+1$. Since $t<2t+1$ for all positive integers $t$, we have $h(G)<|E(G)|$. 
\end{proof}

The \textit{line graph} (see \cite{GY}) of a given graph $G$, denoted by $L(G)$, has the edges of $G$ as vertices with two vertices in $L(G)$ are adjacent if, as the edges of $G$, they are adjacent in $G$. The following theorem establishes a relation between the primitive hole numbers of a graph and its line graph.

\begin{thm}\label{T-HNLG1}
For a given graph $G$ and its line graph $L(G)$, $h(G)\le h(L(G)).$
\end{thm}
\begin{proof}
Note that $L(G)$ contains a triangle if three or more edges are incident to a vertex in $G$. In this context, we need to verify the following possible cases.

\vspace{0.2cm}

\ni {\em Case-1:} First assume that $\Delta(G)\le 2$. Then, $G$ is either a path or a cycle. If $G\cong P_n$, a path on $n$ vertices, then $L(G)=P_{n-1}$. In this case, $h(G)=h(L(G))=0$. If $G\cong C_n$, then $L(G)\cong G=C_n$. If $n=3$, then both $G$ and $L(G)$ themselves are primitive holes and for $n\ge 4$, $G$ and $L(G)$ do not contain primitive holes. In all these cases, $h(G)=h(L(G))$.

\vspace{0.2cm}

\ni {\em Case-2:} If $G$ is a tree with $\Delta(G)\ge 3$, then one vertex, say $v$, in $G$ with $d(v)\ge 3$. Then, the vertices of $L(G)$ corresponding to the edges of $G$ incident on the vertex $V$ $G$ are mutually adjacent in $L(G)$ and hence it corresponds to a complete subgraph $K_{d(v)}$ of the graph $L(G)$. More over, every vertex with degree greater than or equal to $3$ contributes to the number of primitive holes in $L(G)$. Let $V'$ be the subset of $V(G)$ containing the vertices of degree greater than or equal to $3$. Clearly, $V'$ is non-empty. Then, for a vertex $v\in V'$, we have $d(v)=n\ge 3$ and hence $\{v\}\cup N(v)\cong K_{1,n}$ and the corresponding induced subgraph of $L(G)$ is $K_n$. Therefore, by Theorem \ref{P-PHN-KN}, this subgraph contains $\binom{n}{3}$ triangles. Hence, the total number of such triangles in $L(G)$ (that do not correspond to triangles in $G$) is $\sum\limits_{v\in V'}\binom{d(v)}{3}$.    
Therefore, in this case, the primitive holes in $L(G)$ is given by $h(L(G))= h(G)+\sum\limits_{v\in V'}\binom{d(v)}{3}$. Therefore, $h(L(G))\ge h(G)$.

\ni Invoking the above mentioned cases, we have $h(G)\le h(L(G))$.
\end{proof}

The \textit{total graph} $T(G)$ (see \cite{MB}) of a graph $G$ is that graph whose vertex set is $V(G)\cup E(G)$ and in which two vertices are adjacent if and only if they are adjacent or incident in $G$.

The following lemma provides a relation between the primitive hole numbers of a graph, its line graph and total graph.

\begin{lem}
For any given graph $G$, $h(G)\le h(L(G)) \le h(T(G))$.
\end{lem}
\begin{proof}
By Theorem \ref{T-HNLG1} , we have $h(G)\le h(L(G))$. The graph $G$ and its line graph $L(G)$ are subgraphs of $T(G)$. Therefore, By Proposition \ref{P-PHN-SG}, we have $h(L(G)) \le h(T(G))$. Combining these two inequalities, we have $h(G)\le h(L(G)) \le h(T(G))$.
\end{proof}

The following theorem establishes an improved lower bound for the primitive hole number of the line graph $T(G)$ of a given graph $G$.

\begin{thm}\label{T-HNTG1}
For any given graph $G$ and its total graph $T(G)$, $|E(G)|\le h(T(G))$.
\end{thm}
\begin{proof}
Let $u, v$ be two adjacent vertices in $G$. Then, the vertices of $T(G)$ corresponding to the elements $u$, $v$ and $uv$ in $G$ form a triangle in $T(G)$. That is, every edge in $G$ corresponds to a triangle in $T(G)$. Then, we have to consider the following cases.

\ni {\em Case-1:} Let $\Delta(G)\le 2$. If $\Delta(G)=0$, then $G\cong P_1$, the trivial graph. Then, $T(G)$ is also a trivial graph and hence $|E(G)|=h(T(G))=0$. If $\Delta(G)=1$, then $G\cong K_2$ and hence $T(G)=C_3$. In this case $|E(G)|=h(T(G))=1$. 

Next, assume that $\Delta(G)=2$. Then, $G\cong P_n$ or $G\cong C_n$, where $n>2$. Let $e_r=v_iv_j$ and $e_s=v_jv_k$ be two adjacent edges in $G$. The vertices in $T(G)$ corresponding to the elements $v_i,v_j$ and the edge $e_r$ (and $v_j,v_k$ and the edge $e_s$) of $G$ form a triangle in $T(G)$. Since $e_r$ and $e_s$ are adjacent in $G$, the vertices in $T(G)$ corresponding to the elements $e_r, e_s$ and the vertex $v_j$ in $G$ also form a triangle in $G$. Now, for every vertex $v_l$, that is adjacent to any one these vertices additionally form two triangles in $T(G)$. Hence, if $G\cong P_n,~n>2$, then $h(G)=2n-3>|E(G)|=n-1$ and if $G\cong C_n$, then $h(G)=2n>|E(G)|=n$. 

\vspace{0.2cm}

\ni {\em Case-2:} Let $\Delta(G)\ge 3$. Then, each edge of $G$ corresponds to a triangle in $G$ and adjacency of two edges also forms a triangle in $T(G)$. Moreover, every $K_{1,3}$ and every $C_3$ in $G$ also form triangles in $T(G)$. Since at least one vertex of $G$ has a degree greater than or equal to $3$, then $h(T(G))>|E(G)|$. Hence, in this case, we have $h(G)<h(T(G))$.

\ni Combining the above two cases, we have $h(G)\le h(T(G))$.
\end{proof}

\ni In view of the above theorem, we can establish the following result.

\begin{cor}
If $G$ is a graph on $3$ or more vertices, then $|E(G)|<h(T(G))$.
\end{cor}

\ni In view of the above theorem, we have the following results.

\begin{thm}\label{T-HNTG2}
For a given graph $G$, $h(T(G))=|E(G)|$ if and only if $G$ has no internal vertices. 
\end{thm}
\begin{proof}
Assume that $G$ has no internal vertex. Then, $\Delta(G)<2$ and either $G\cong K_1$ or $G\cong K_2$. Hence, by \ref{T-HNTG1}, $h(T(G))=|E(G)|$. 

Assume the converse. That is, for a given graph $G$, we have $h(T(G))=|E(G)|$. If possible, let $G$ has some internal vertex, say $v$. Then, $v$ is adjacent to at least two vertices $u$ and $w$ in $G$. Let $e_i=uv$ and $e_j=vw$. Label the vertices of $T(G)$ by the same label of the corresponding element of $G$. Then, we can find three triangles in $T(G)$, formed by the vertex  triplets $\{u,v, e_i\}$, $v,w, e_j$ and $e_i,e_j, v$. Hence, for any pair of adjacent edges in $G$, there exists three triangles in $G$ incident on their common (internal) vertex. Therefore, $|E(G)|<h(T(G))$, a contradiction to the hypothesis. Therefore, $G$ can not have internal vertices.
\end{proof}

\ni The above theorem leads us to the following result.

\begin{cor}
For a given graph $G$, $h(T(G))>|E(G)|$ if and only if $G$ has some internal vertices. 
\end{cor}
\begin{proof}
The statement of the theorem is the contrapositive of Theorem \ref{T-HNTG2}.
\end{proof}

\section{The Primitive Hole Number of the Underlying Graph of a Jaco Graph}

Let us now recall the definition of Jaco graphs, finite and infinite, as follows.

\begin{defn}{\rm 
\cite{KFW} The {\em infinite Jaco graph}, denoted by $J_{\infty}(1)$, is a directed graph with vertex set $V$ and arc set $A$ such that $V(J_\infty(1)) = \{v_i| i \in \N\}$, $E(J_\infty(1)) \subseteq \{(v_i, v_j): i, j \in \N, i< j\}$ and $(v_i,v_ j) \in A(J_\infty(1))$ if and only if $2i - d^-(v_i)\ge j$. A \textit{finite Jaco graph}, denoted by $J_n(1)$, is a finite subgraph of $J_{\infty}(1)$, where $n$ is a finite positive integer.}
\end{defn}

\begin{defn}{\rm 
The vertices attaining degree $\Delta (J_n(1))$ is called the \textit{Jaconian vertices} of the Jaco Graph $(J_n(1)$. The set of Jaconian vertices of $(J_n(1)$ is denoted by $\mathbb{J}_n(1)$.}
\end{defn}

\ni The Jaco graph $(J_\infty(1))$, with the vertex set $V(J_\infty(1)) = \{v_i: i \in \N\}$,  has the fundamental properties as given below.

\begin{enumerate}\itemsep=0mm
\item if $v_j$ is the head of an arc $(v_i, v_j)$, then $i<j$,
\item if $k\in \N$ is the smallest integer such that $v_k$ is a tail vertex of an arc $(v_k,v_j)$ in $J_{\infty}(1)$, then for all $k< l <j$, the vertex $v_ l$ is the tail of an arc to $v_j$.
\item the degree of vertex $k$ is $d(v_k) = k$. 
\end{enumerate}

The family of finite directed graphs are those limited to $n \in \N$ vertices by lobbing off all vertices (and hence the arcs incident on these vertices) $v_t,~ t > n$. Hence, trivially we have $d(v_i) \leq i$ for $i \in \N$.

\vspace{0.2cm}

\ni We denote the underlying graph by $J^{\ast}_n(1)$. We now provide a recursive formula of the number of primitive holes, $h(J^{\ast}_{n+1}(1))$ in terms of $h(J_n^{\ast}(1)),n \ge 4$. 

\vspace{0.2cm}

\ni If $v_i$ is the prime Jaconian vertex of a Jaco Graph $J_n(1)$, the complete subgraph on vertices $v_{i+1}, v_{i+2},v_{i+3}, \ldots, v_n$ is called the {\em Hope subgraph} of a Jaco Graph and denoted by $\mathbb{H}_n(1)$.

In view of the above definitions and concept, we can determine the primitive hole number of the underlying graph of a Jaco graph as in the following theorem.

\begin{thm}
Let $J^{\ast}_n(1)$ be the underlying graph of a finite Jaco Graph $J_n(1)$ with Jaconian vertex $v_i$, where $n$ is a positive integer greater than or equal to $4$. Then, $h(J^{\ast}_{n+1}(1)) = h(J^{\ast}_n(1)) + \sum\limits_{j=1}^{(n-i)-1}(n-i)-j$.
\end{thm}
\begin{proof}
Consider the underlying Jaco graph, $J^{\ast}_n(1), n \in \N,n\ge 4$ with prime Jaconian vertex $v_i$. Now consider $J^{\ast}_{n+1}(1)$. From the definition of a Jaco graph, the extension from $J^{\ast}_n(1)$ to $J^{\ast}_{n+1}(1)$ is obtained by adding the vertex $v_{n+1}$ and the edges $v_{i+1}v_{n+1}, v_{i+2}v_{n+1},\ldots, v_nv_{n+1}$ to $J^{\ast}_n(1)$.
	
We know that the Hope graph $\mathbb{H}(J^{\ast}_n(1))$ (see \cite{KFW}) is the complete graph on vertices $v_{i+1}, v_{i+2}, v_{i+3},\ldots, v_n$. So it follows that the triplets of vertices induce the additional primitive holes.
\begin{align*}
& \underbrace{\{v_{i+1}, v_{n+1}, v_{i+2}\}, \{v_{i+1}, v_{n+1}, v_{i+3}\},\ldots, \{v_{i+1}, v_{n+1}, v_n\}}_{(n-i)-1, sets},\\
& \underbrace{\{v_{i+2}, v_{n+1}, v_{i+4}\}, \{v_{i+2}, v_{n+1}, v_{i+5}\},\ldots, \{v_{i+2}, v_{n+1}, v_n\}}_{(n-i)-2, sets},\\
& \underbrace{\{v_{i+3}, v_{n+1}, v_{i+4}\}, \{v_{i+3}, v_{n+1}, v_{i+5}\},\ldots, \{v_{i+3}, v_{n+1}, v_n\}}_{(n-i)-3, sets},\\
& ~~~~~~~~~~~\vdots \\
& ~~~~~~~~~~~\vdots \\
& \underbrace{\{v_{i+(n-i -1)}, v_{n+1}, v_{i+ (n-i)}\}}_{(n-i)-(n-i-1)=1, set}
\end{align*}

\ni Therefore, we have $h(J^{\ast}_{n+1}(1))=h(J^{\ast}_n(1)) +\sum\limits_{j=1}^{(n-i)-1}(n-i)-j$.
\end{proof} 

\section{Primitive Degree of Graphs}

In this section, we introduce the notion of the primitive degree of a vertex of a given graph $G$ as follows.

\begin{defn}{\rm 
The \textit{primitive degree} of a vertex $v$ of a given graph $G$ is the number of primitive holes to which the vertex $v$ is a common vertex. The primitive degree of a vertex $v$ of $G$ is denoted by $d^p_G(v)$.}
\end{defn}

It follows easily that  the primitive degree of each vertex of the complete graph $K_3$ is $1$. We can also note that the primitive degree of any vertex of a path $P_n$ and primitive degree of any vertex of a cycle $C_n;~n\ge 3$ is $0$. For the complete graphs $K_n; ~ n\ge 4$, we have the following result.
	
\begin{thm}\label{T-PDKN1}
The primitive degree of a vertex $v$ of a complete graph $K_n;~n\ge 3$ is $\sum\limits_{i=1}^{n-2}i$.
\end{thm}
\begin{proof}
Let $v$ be an arbitrary vertex of a complete graph $K_n$. Then, any two pair of edges of $K_n$ form a triangle in $K_n$ together with an edge that is not incident on $v$. That is, the number of triangles incident on the vertex $v$ is the number of distinct pairs of edges that are incident on $v$. That is, $d_{K_n}^p(v)=\binom{n-1}{2} = \frac{1}{2}(n-1)(n-2)=\sum\limits_{i=1}^{n-2}i$.
\end{proof}

Now, we also give an alternate proof for Theorem \ref{T-PDKN1} using mathematical induction as follows.

\begin{proof}
Consider the complete graph $K_3$ first. Here, in this case, $d^p_{K_3}(v)= 1$. Assume the result holds for $K_t$. Hence, we have $d^p{K_t}(v)=\sum\limits_{i=1}^{t-2}i$. Label the vertices of $K_t$ as $v_1, v_2, v_3,\ldots, v_t$. 

Now, consider the complete graph $K_{t+1}$. In the extension from $K_t$ to $K_{t+1}$, the vertex $v_{t+1}$ with the edges $v_{t+1}v_1, v_{t+1}v_2, v_{t+1}v_3,\ldots, v_{t+1}v_t$ were added. Hence, the following combinations of vertex sets induce triangles in $K_{t+1}$. 
\begin{align*}
&\underbrace{\{v_{t+1}, v_1, v_2\}, \{v_{t+1}, v_1, v_3\},\ldots, \{v_{t+1}, v_1, v_t\}}_{d_{K_t}(v_1), sets},\\
&\underbrace{\{v_{t+1}, v_2, v_3\}, \{v_{t+1}, v_2, v_4\},\ldots, \{v_{t+1}, v_2, v_t\}}_{d_{K_t}(v_2)-1, sets},\\ 
&\underbrace{\{v_{t+1}, v_3, v_4\}, \{v_{t+1}, v_3, v_5\},\ldots\ldots, \{v_{t+1}, v_3, v_t\}}_{d_{K_t}(v_3)-2, sets},\\
&~~~~~~~~~~~~~\vdots\\
&~~~~~~~~~~~~~\vdots \\
&\{v_{t+1}, v_{t-1}, v_t\}
\end{align*}
induce the primitive holes having $v_{t+1}$ as a common vertex, exhaustively. The total of such sets which induce triangles in $K_{t+1}$ is given by $(t-1)+(t-2)+(t-3)+\ldots +1=((t+1)-2)+\sum\limits_{i=1}^{t-2}i = \sum\limits_{i=1}^{(t+1)-2}i$.
	
\ni Hence, the result follows by induction.
\end{proof}

A question that arouses much interest in this context is about the primitive degree of the vertices of a line graph of a given graph. The primitive degree of the vertices of line graphs is determined in following theorem.

\begin{thm}
The primitive degree of a vertex $v$ in the line graph of a graph $G$ is $d^p_{L(G))}=\varLambda + \binom{d(v_i)-1}{2}+ \binom{d(v_j)-1}{2}$, where $\varLambda$ is the number of triangles containing the edge $e=v_iv_j$ of $G$ corresponding to the vertex $v$ in $L(G)$. 
\end{thm} 
\begin{proof}
A triangle in $L(G)$ corresponds to either a triangle in $G$ or a $K_{1,3}$ in $G$. The number of triangles incident on a vertex $v$ of $L(G)$ is the number triangles or the number of $K_{1,3}$ which contain the edge $e$ of $G$, corresponding to the vertex $v$ in $L(G)$. The number of $K_{1,3}$ in $G$ containing the edge $e=v_iv_j$ is $\binom{d(v_i)-1}{2}+ \binom{d(v_j)-1}{2}$. Therefore, the number of triangles incident on the vertex $v$ is $\binom{d(v_i)-1}{2}+ \binom{d(v_j)-1}{2}+\varLambda$, where $\varLambda$ is the number of triangles in $G$ containing the edge $e$ in $G$. This completes the proof. 
\end{proof} 

The following result establishes recurrence relation on the primitive degree of the vertices of complete graphs.

\begin{prop}
For the graph $K_n;~n\ge 4$, the primitive degree $d^p_{K_n}(v)= d^p_{K_{n-1}}(u)+(n-2)$.
\end{prop}
\begin{proof}
We know that the primitive degree of every vertex of a complete graph is the same. Now assume that we have the primitive degree of a vertex, say $v_{i}$ in $K_{n-1}$, where $1\le i\le n-1$. Now, extend $K_{n-1}$ to $K_n$ by adding a vertex, say $v_n$ to $K_{n-1}$ and joining every vertex of $K_{n-1}$ to the new vertex $v_n$. Then, the vertex triplets $\{v_{i}, v_n, v_j\}$, $i\ne j,~ 1\le i\ne j\le n-1$ form new triangles in $K_n$ that are not in $K_{n-1}$. Hence, exactly $(n-2)$ additional primitive holes incident on the vertex $v_i$ in $K_n$ than the number of primitive holes in $K_{n-1}$.  
\end{proof}

Invoking the concepts mentioned above, the primitive degree of a Jaco graph is determined in the following result.

\begin{thm}
For $n\ge 5$, the primitive hole number of the Jaco graph $J_n(1)$ is $h(J_n(1))=\sum (d^+_{J_n(1)}(v_j)-1)$, for all $d^+_{J_n(1)}(v_j)\ge 2$.
\end{thm}
\begin{proof}
We prove this theorem bu mathematical induction. Consider $J^{\ast}_5(1)$. It has one primitive hole induced by the vertices $\{v_3, v_4, v_5\}$. Only the vertex $v_3$ has $d^+_{J_5(1)}(v_3)=2\ge 2$. Since $\sum\limits_{d^+_{J_5(1)}(v_3)\ge 2} (2-1)=1$, the results holds for $J^{\ast}_5(1)$.

Assume the result holds for $J^{\ast}_k(1)$ having Jaconian vertex $v_i$. Hence, we have
$h(J^{\ast}_k(1)) = \sum\limits(d^+_{J_k(1)}(v_j)-1), ~ \forall d^+_{J_k(1)}(v_j)\ge 2$.
Now, consider the graph $J^{\ast}_{k+1}(1)$. We note that $h(J^{\ast}_{k+1}(1)) = h(J^{\ast}_k(1)) + d^p_{J^{\ast}_{k+1}(1)}(v_{k+1})$. Consider the number of primitive holes having both $v_{i+1}, v_{k+1}$ in common. Clearly, the triangles induced by $\{ v_{k+1}, v_{i+1}, v_{i+2}\}, \{ v_{k+1}, v_{i+1}, v_{i+3}\}, \{ v_{k+1}, v_{i+1}, v_{i+4}\}, ..., \{ v_{k+1}, v_{i+1}, v_k\}$ are exhaustive. The number of primitive holes having vertices $v_{i+1}, v_{k+1}$ in common is given by $d^+_{J_k(1)}(v_{i+1}) = d^+_{J_{k+1}(1)}(v_{i+1}) - 1$. Since the latter sub-result applies to every vertex $v_{i+2}, v_{i+3}, v_{i+4}, \ldots, v_k, v_{k+1}$, the result $h(J^{\ast}_n(1))=\sum\limits(d^+_{J_n(1)}(v_j) -1), ~ \forall d^+_{J_n(1)}(v_j)\ge 2$, follows by induction.
\end{proof}

In the following proposition, the primitive hole number of a complete graph $K_n$ is determined recursively from the primitive degree of the vertices of complete graphs order less than or equal to $n$. 

\begin{prop}
For a complete graph $K_n,n\ge 4$, $h(K_n)=\sum\limits_{i=3}^{n}d^p_{K_i}(v)$, where ${v \in V(K_i)}$.
\end{prop}
\begin{proof}
Consider a complete graph $K_n,n\ge 4$ and label its vertices as $v_1, v_2, v_3, ..., v_n$. By Theorem \ref{T-PDKN1}, $d^p_{K_n}(v)$, where $v \in V(K_n)= \sum\limits_{i=1}^{n-2}i$, for all $v_i \in V(K_n)$. Hence, without loss of generality, we determine $d^p_{K_n}(v_1)$, thereafter $d^p_{K_n - v_1}(v_2)$, thereafter $d^p_{K_n - \{v_1, v_2\}}(v_3)$ and so on until we obtain the triangle on vertices $v_{n-2}, v_{n-1}, v_n$, which has $h(K_3) = 1$. Clearly, the summation of this iterative procedure equals the total number of primitive holes, $h(K_n)$. Hence, the result $h(K_n) = \sum\limits_{i=3}^{n}d^p_{K_i}(v)$, holds.
\end{proof}

A natural and relevant question in this context is whether there is a relation between the primitive hole number of a graph $G$ and the primitive degree of its vertices. The following theorem gives an answer to this question.

\begin{thm}
For a simple connected graph $G$, we have $h(G)=\frac{1}{3}\sum\limits_{v \in V(G)} d^p_{G}(v)$.
\end{thm}
\begin{proof}
Consider any simple connected graph $G, h(G) \ge 0$. If $h(G) = 0$, it implies that $d^p_{G}(v) = 0,\forall v \in V(G)$. Hence, the result holds. If $h(G)\ge 1$, say $h(G) = t$, then label the primitive holes $K^1_3, K^2_3, K^3_3, ..., K^t_3,$ respectively. Each primitive hole $K^i_3,1 \leq i \leq t$ in $G$ has the vertices say, $u, v, w$. In determining $d^p_{G}(u), d^p_{G}(v)$ and $d^p_{G}(w)$, the primitive hole $K^i_3$ is counted three times. Since this triple count applies to all primitive holes the result, we have $h(G) = \frac{1}{3}\sum d^p_{G}(v),\forall v \in V(G)$.
\end{proof}

Invoking the results on the primitive hole number of complete graphs and on the primitive degree of the vertices of those complete graph, we have the following theorem

\begin{thm}
For the Jaco graph $J_n(1),n\ge 5$ having the Jaconian vertex $v_i$, we have $h(J^{\ast}_n(1)) = \binom{n-i}{3} + \sum\limits(d^+_{J_n(1)}(v_j)-1)$, where $d^+_{J_n(1)}(v_j)\ge 2; j\leq i$.
\end{thm}
\begin{proof}
For the Jaco graph $J^{\ast}_n(1), n\ge 5$ having the Jaconian vertex $v_i$, we apply Theorem \ref{T-PDKN1} over vertices $v_1, v_2, v_3,\ldots, v_i$ followed by applying Proposition \ref{P-PHN-KN} to the Hope graph $\mathbb{H}(J^{\ast}_n(1))$.
\end{proof}

\section{Conclusion}
We have discussed the primitive hole number of graphs and primitive degree of vertices of certain simple connected graphs. The study seems to be promising as it can be extended to certain standard graph classes and certain graphs that are associated with the given graphs. Determining the primitive degree of a vertex of the total graph of a graph $G$ is an open problem. More problems in this area are still open and hence there is a wide scope for further studies. 

\section*{Acknowledgements}

The authors gratefully acknowledge the comments and suggestions of Mrs. K. P. Chithra, wife of the second author, which improved the overall presentation of the content of this paper.

\end{document}